\documentclass[12]{article}
\usepackage{amsfonts}
\usepackage{amsmath}
\usepackage{amssymb}

\def \dbigcup {\displaystyle\bigcup }
\def \dsum {\displaystyle\sum }
\def \limits { }

\newtheorem{theorem}{Theorem}

\newtheorem{lemma}[theorem]{Lemma}

\newenvironment{proof}[1][Proof]{\noindent\textbf{#1.} }{\ \rule{0.5em}{0.5em}}

\input{tcilatex}

\begin{document}

\title{A Combinatorial Problem Related to Sparse Systems of Equations}
\author{Peter Horak $^{1}$\thanks{%
Research of P. Horak was supported in part by a grant from SIAS, University
of Washington, Tacoma.} \ \ Igor Semaev $^{2}$\thanks{%
Research of P. Horak and I. Semaev was supported in part by a grant SPIRE
program in 2013-2015 from University of Bergen. Research of I. Semaev was
also partly supported by the EEA Grant SK06-IV-01-001, and the state budget
of the Slovak Republic from the EEA Scholarship Programme Slovakia.} \ \
Zsolt Tuza $^{3,4}$\thanks{%
Research of Zs.~Tuza was supported in part by the grant T\'{A}%
MOP-4.2.2.B-15/1/KONV-2015-0004.} \\
{\normalsize $^{1}$ School of Interdisciplinary Arts \& Sciences, University
of Washington,}\\
{\normalsize Tacoma, USA }\\
$^{2}$ {\normalsize Department of Informatics, University of Bergen, Norway}%
\\
{\normalsize $^{3}$ Alfr\'{e}d R\'{e}nyi Institute of Mathematics, Hungarian
Academy of Sciences,}\\
{\normalsize Budapest, Hungary }\\
{\normalsize $^{4}$ University of Pannonia, Veszpr\'{e}m, Hungary }}
\date{}
\maketitle

\begin{abstract}
Nowadays sparse systems of equations occur frequently in science and
engineering. In this contribution we deal with sparse systems common in
cryptanalysis. Given a cipher system, one converts it into a system of
sparse equations, and then the system is solved to retrieve either a key or
a plaintext. Raddum and Semaev proposed new methods for solving such sparse
systems. It turns out that a combinatorial MaxMinMax problem provides bounds
on the average computational complexity of sparse systems. In this paper we
initiate a study of a linear algebra variation of this MaxMinMax problem.
\end{abstract}

\bigskip


\textbf{Keywods}: sparse systems of equations; MaxMinMax problem; gluing
algorithm.

\section{Introduction}

Sparse objects such as sparse matrices and sparse systems of (non-)linear
equations occur frequently in science and engineering. Nowadays sparse
systems are often studied in algebraic cryptanalysis as well. First, given a
cipher system, one converts it into a system of equations. Second, the
system of equations is solved to retrieve either a key or a plaintext. As
pointed out in \cite{C}, this system of equations will be sparse, since
efficient implementations of real-word systems require a low gate count.

There are plenty of papers on methods for solving a sparse system of
equations

\begin{equation}
f_{i}(X_{i})=0(1\leq i\leq m)  \label{AA}
\end{equation}%
over $GF(q).$ The worst case complexity bounds on (\ref{AA}) are attained in
the case of sparse systems describing the SAT problem. These bounds are
exponential with respect to the number of unknowns in (\ref{AA}). For
example, in the case of the binary field, the bounds are $2^{cn},$ where the
constant $c$ is close to $1,$ and depends on the size of $\left\vert
X_{i}\right\vert ^{\prime }$s, see \cite{I}. In \cite{Igor2} the so-called
Gluing Algorithm was designed to solve such systems over any finite field $%
GF(q)$. If the set $S_{k}$ of solutions of the first $k$ equations together
with the next equation $f_{k+1}=0$ is given then the algorithm constructs
the set $S_{k+1}.$ It is shown there that the average complexity of finding
all solutions to the original system is $O(mq^{\max \left\vert \cup
_{1}^{k}X_{j}\right\vert -k}),$ where $m$ is the total number of equations,
and $\cup _{1}^{k}X_{j}$ is the set of all unknowns actively occurring in
the first $k$ equations. Clearly, the complexity of finding all solutions to
the system by the Gluing Algorithm depends on the order of equations. Hence
one is interested to find a permutation $\pi $ that minimizes the average
complexity, and also to describe the worst-case scenario, i.e., the system
of equations for which the average complexity of the method is maximum.
Therefore, Semaev \cite{Igor} suggested to study the following combinatorial
MaxMinMax problem.

Let $\mathcal{S}_{n,m,t}$ be the family of all collections of sets $\mathcal{%
X=}\{X_{1},...,X_{m}\},$ where the $X_{i}$ are subsets of an underlying $n$%
-set $X$, and $\left\vert X_{i}\right\vert \leq t$ holds for all $i\in
\lbrack m]$\thinspace ;\thinspace\ we allow that some set may occur in $%
\mathcal{X}$ more than once. Then we define 
\begin{equation}
f_{t}(n,m):=\max_{\mathcal{X}}\min_{\pi }\max_{1\leq k\leq m}(\left\vert
\dbigcup\limits_{i=1}^{k}X_{\pi (i)}\right\vert -k)  \label{1}
\end{equation}%
where the minimum runs over all permutations $\pi $ on $[m],$ and the
maximum is taken over all families $\mathcal{X}$ in $\mathcal{S}_{n,m,t}.$

In \cite{HT} the authors confined themselves to the case $\left\vert
X_{i}\right\vert \leq 3$ for all $i\in \lbrack m]$. It was shown there that,
for $n\geq 2$ and all $m\leq n-1,$ $f_{2}(n,m)$ equals the maximum number of
non-trivial components in a simple forest on $n$ vertices with $m$ edges;
otherwise $f_{2}(n,m)=1.$ The main result of that paper claims that $%
f_{3}(n,n)$ grows linearly. More precisely, the following estimates are
valid.

\begin{theorem}
\label{bounds} For all\/ $n$ sufficiently large, $f_{3}(n,n)\geq \frac{n}{%
12.2137}$ holds, while\/ $f_{3}(n,n)\leq \left\lceil \frac{n}{4}\right\rceil
+2 $ for all\/ $n\geq 3$.
\end{theorem}

Later, an asymptotically better upper bound was proved in \cite{Igor}; we
note that the proof of the bound is algorithmic, and the needed permutation $%
\pi $ is constructed in polynomial time.

\begin{theorem}
For all\/ $n$, $f_{3}(n,n)\leq \frac{n}{5.7883}+1+2\log _{2}n$.
\end{theorem}

As a corollary we get: Let $\mathcal{X}$ be fixed. If $|X_{i}|\leq 3$, $m=n$%
, then the average complexity of finding all solutions in $GF(q)$ to
polynomial equation system (\ref{AA}) is at most $q^{\frac{n}{5.7883}+O(\log
n)}$ for arbitrary $\mathcal{X}$ and $q$. \bigskip

In \cite{RS} a new method for representing and solving systems of algebraic
equations common in cryptanalysis has been proposed. This method differs
from the others in that the equations are not represented as multivariate
polynomials, but as a system of Multiple Right-Hand Sides (MRHS) linear
equations. The results overcome significantly what was previously achieved
with Gr\"{o}bner Basis related algorithms. We point out that equations
describing the Data Encryption Standard (DES) \cite{R} and the Advanced
Encryption Standard (AES) can be expressed in MRHS form as well.

AES is likely the most commonly used symmetric-key cipher; AES became
effective as a US federal government standard on May 26, 2002 after approval
by the Secretary of Commerce. It is the first publicly accessible and open
cipher approved by the National Security Agency (NSA) for top secret
information when used in an NSA-approved cryptographic module. An
application of MRHS equations enables one to improve on the linear
cryptanalysis of DES \ \cite{iS14}.

Let $X$ be a column $n$-vector of unknowns over $GF(q)$. Then an MRHS system
of equations is a system of inclusions 
\begin{equation}
A_{i}X\in \{b_{i_{1}},\ldots ,b_{i_{s_{i}}}\},\quad i=1,\ldots ,m,
\label{system}
\end{equation}%
where the $A_{i}$ are matrices over $GF(q)$ of size $t_{i}\times n$ and of
rank $t_{i}$, and the $b_{ij}$ are column vectors of length $t_{i}$. An $%
X=X_{0}$ is a solution to (\ref{system}) if it satisfies all inclusions in (%
\ref{system}). Methods to solve an arbitrary MRHS system of equations were
introduced in [6] as well. \ 

One of the main goals of our paper is to get asymptotic bounds on the
average complexity of solving (\ref{system}). As noted by Semaev, such
bounds can be obtained by studying a generalization of the combinatorial
problem described in (\ref{1}). {In particular, when matrices $A_{i}$ are
fixed and the right hand side columns in \textup{(\ref{system})} are
generated according to a probabilistic model presented in the next section,
we will prove }

\begin{theorem}
\label{Theorem_3} Let\/ $t_{i}\leq t$ for some fixed\/ $t$ and\/ $n$ tends
to infinity. Then the average complexity of solving\/ \textup{(\ref{system})}
is 
\begin{equation*}
O(m\,q^{n-\lceil n/t\rceil }).
\end{equation*}
\end{theorem}

\begin{proof}
The proof of the statement follows directly from Theorem \ref{Th_1} and
Lemma \ref{lemma} that will be stated and proved in Section 2 and Section 3,
respectively.
\end{proof}

\section{Probabilistic Model}

We will use a generalization of the probabilistic model used in \cite{Igor}.
For $i=1,...,m$ we choose in random, independent and uniform way polynomials 
$g_{i}(y_{1},\ldots ,y_{t_{i}})$ over $GF(q).$ The degree of the polynomial
in each of its variables is at most $q-1$. The right-hand sides $%
\{b_{i_{1}},\ldots ,b_{i_{s_{i}}}\}$ in \textup{(\ref{system})} are zeroes
of polynomials $g_{i}(y_{1},\ldots ,y_{t_{i}})$ over $GF(q).$ So that {%
\textup{(\ref{system})}} is equivalent to a system of polynomial equations

\begin{equation*}
g_{i}(A_{i}X)=0,i=1,...,m.
\end{equation*}

\begin{lemma}
\label{c1}The probability of a fixed vector\/ $b\in GF(q)^{t}$ to be a zero
of a random polynomial\/ $g$ over\/ $GF(q)$ equals\/ $\frac{1}{q}.$
\end{lemma}

\begin{proof}
The set of the polynomials in $t$ variables over $GF(q)$ and of degree at
most $q-1$ in each of its variables is in one-to-one natural correspondence
with the set of all mappings from $GF(q)^{t}$ to $GF(q)$. The sought
probability equals $\frac{1}{q}$ as the overall number of such polynomials
(mappings) is $q^{q^{t}},$ and the number of the polynomials (mappings) $g$,
where $g(b)=0$, is $q^{q^{t}-1}$.
\end{proof}

\medskip

A random selection of elements from a set of cardinality $n$ with
probability $p$, yields a set whose expected cardinality is equal to $pn$.
In particular, for our problem under consideration we immediately obtain:

\begin{lemma}
\label{Lemma_01} Let\/ $AX\in S$ be a multiple right-hand side equation,
where\/ $A$ is of size\/ $t\times n$ and of rank\/ $t$. Suppose the
right-hand sides\/ $S$ are taken randomly such that\/ $\mathbf{Pr}(b\in S)=p$%
. Then the average size of\/ $S$ is\/ $p\,q^{t}$.
\end{lemma}

For two equations $A_{1}X\in S_{1}$ and $A_{2}X\in S_{2}$, one can construct
an equation $AX\in S$ such that its solutions are precisely the common
solutions to the original two equations. In [6] the operation is called
gluing, this is a linear algebra generalization of gluing earlier introduced
in [7]. As above let $A_{i}$ be of size $t_{i}\times n$ and of rank $t_{i}$.
The matrix $A$ is constructed by writing $A_{2}$ under $A_{1}$ and
eliminating dependent rows. Therefore, $t=\mathbf{rank}(A)=\mathbf{rank}%
(A_{1},A_{2})$ is also the number of rows in $A$. Assume that $S_{1},S_{2}$
are randomly generated.

\begin{lemma}
\label{Lemma_0} Let\/ $p_{i}=\mathbf{Pr}(b_{i}\in S_{i})$ for any column
vectors\/ $b_{i}$ of size\/ $t_{i}$. If the sets\/ $S_{1}$ and\/ $S_{2}$ are
generated independently, then for any column vector\/ $b$ of size\/ $t$ we
have\/ $\mathbf{Pr}(b\in S)=p_{1}p_{2}$.
\end{lemma}

\begin{proof}
Consider the system of linear equations $A_{1}X=b_{1},A_{2}X=b_{2}$. We can
express them as 
\begin{equation*}
\left( 
\begin{array}{c}
A_{1} \\ 
A_{2}%
\end{array}%
\right) X=\left( 
\begin{array}{c}
b_{1} \\ 
b_{2}%
\end{array}%
\right)
\end{equation*}%
If the system is consistent, then it is equivalent to $AX=b$ for some $b$.
Due to the equivalence, $b\in S$ is constructed from only one pair $b_{1}\in
S_{1}$ and $b_{2}\in S_{2}$. This proves the statement.
\end{proof}

\bigskip

The lemma implies that the average number of the right-hand sides (the size
of $S$) in $AX\in S$ is 
\begin{equation*}
p_{1}p_{2}\,q^{\mathbf{rank}(A_{1},A_{2})}.
\end{equation*}%
We say the system \textup{(\ref{system})} is solved if it is represented by
only one equation $AX\in S$. In turn, this is equivalent to solving $|S|$
systems of ordinary linear equations over $GF(q)$; its complexity is
neglected here.

Let $r_{t}$ denote the rank of all row vectors in $A_{1},A_{2},\ldots ,A_{t}$%
.

\begin{theorem}
\label{Th_1}Assume that the right-hand sides of the equations in\/ \textup{(%
\ref{system})} are generated in a random, independent, and a uniform way.
Then the average complexity of\/ \textup{(\ref{system})} is 
\begin{equation*}
O(m\max_{t}q^{r_{t}-t})
\end{equation*}
\end{theorem}

\begin{proof}
The system is solved in aggregate by at most $m-1$ applications of gluing
operation. The complexity of one gluing is proportional to the number of
right-hand sides in the equations to glue and the number of the resulting
right-hand sides, see \cite{RS}. By Lemma \ref{c1} we know that a column
vector appears on the right-hand side in \textup{(\ref{system})} with
probability $1/q$. Further, by Lemmas \ref{Lemma_01} and \ref{Lemma_0},
after the $t$-th gluing the average number of the right-hand sides is $\frac{%
1}{q^{k}}\,q^{\mathbf{rank}(A_{1},\ldots ,A_{k})}$. The average complexity
of the system is then the sum of the average number of the right-hand sides
after each application: 
\begin{equation*}
\sum_{k=1}^{m}\frac{1}{q^{t}}\,q^{\mathbf{rank}(A_{1},\ldots ,A_{t})}\leq
m\max_{t}q^{r_{t}-t}.
\end{equation*}%
That proves the statement.
\end{proof}

\section{Corresponding Combinatorial Problem}

In this section we formulate a combinatorial problem related to an MRHS
system of equations. It turns out that the complexity of the problem can be
described by a generalization of the function $f_{t}(n,m)$ defined in (\ref%
{1}), namely the size of the union of the first $k$ sets is replaced by the
rank of vectors belonging to the first $k$ matrices.\ Formally, let $%
\mathcal{S}_{n,m,t,V}$ be the family of all collections of sets of $\mathbf{%
vectors}$ $\mathcal{X=}\{X_{1},...,X_{m}\}$ in an $n$-dimensional vector
space $V,$ over any finite on infinite field, under the restriction $%
\left\vert X_{i}\right\vert \leq t$ for all $i\in \lbrack m].$ We set 
\begin{equation}
F_{t}(n,m):=\max_{\mathcal{X}}\min_{\pi }\max_{1\leq k\leq m}(\mathbf{rank}%
\!\!\left( \dbigcup\limits_{i=1}^{k}X_{\pi (i)}\right) -k),  \label{3}
\end{equation}%
where the minimum runs over all permutations $\pi $ on $[m],$ and the
maximum is taken over all families $\mathcal{X}$ in $\mathcal{S}_{n,m,t,V}.$
We note that the definition of the function $F_{t}(n,m)$ reflects the fact
that the order of matrices $A_{i}^{\prime }$s is important in the Gluing
Algorithm.

Although functions $f_{t}(n,m)$ and $F_{t}(n,m)$ are defined in a similar
way, their behavior is dramatically different.

We start with a rather simple upper bound, used in the proof of Theorem \ref%
{Theorem_3}.

\begin{lemma}
\label{lemma}Let\/ $n,t$ be natural numbers. Then 
\begin{equation}
F_{t}(n,n)\leq n-\left\lceil \frac{n}{t}\right\rceil  \label{up}
\end{equation}
\end{lemma}

\bigskip

\begin{proof}
Let $n=st+k,$ where $0\leq k<t$. We have 
\begin{equation*}
\mathbf{rank}(X_{1},\ldots ,X_{i})-i\leq i(t-1)\leq s(t-1)=n-k-s\leq
n-\lceil n/t\rceil
\end{equation*}%
for $i\leq s$. Moreover, 
\begin{equation*}
\mathbf{rank}(X_{1},\ldots ,X_{i})-i\leq n-\lceil n/t\rceil
\end{equation*}%
for $i\geq s+1$, as $i\geq \lceil n/t\rceil $ in this case, and the upper
bound (\ref{up}) follows.
\end{proof}

\bigskip

It turns out that bounds on $F_{t}(m,n)$ constitute a challenge for most
vector spaces $V$ even for $t=2$,$\,\ $although in the case of the function $%
f_{t}(n,m)$ we know its exact value in the case of $t=2$. Surprisingly, the
upper bound (\ref{up}) can be attained in many cases.

\begin{theorem}
Let\/ $F$ be any infinite field, or\/ $F=GF(q),$ the finite field with\/ $%
q\geq tn$ elements. Then for any\/ $n$ and\/ $t$ we have\/ $%
F_{t}(n,n)=n-\left\lceil \frac{n}{t}\right\rceil .$
\end{theorem}

\begin{proof}
By (\ref{up}) it suffices to bound $F_{t}(n,n)$ from below. Let 
\begin{equation*}
M=\left( 
\begin{array}{cccc}
1 & \alpha _{1} & \ldots & \alpha _{1}^{n-1} \\ 
1 & \alpha _{2} & \ldots & \alpha _{2}^{n-1} \\ 
\ldots &  &  &  \\ 
1 & \alpha _{tn} & \ldots & \alpha _{tn}^{n-1}%
\end{array}%
\right) .
\end{equation*}%
be a matrix of size $tn\times n$, where $\alpha _{i},i=1,\ldots ,tn$ are
pairwise different elements from the field $F.$ As $M$ is a Vandermonde
matrix, any $n$ rows of $M$ are linearly independent. To construct a desired
collection $\mathcal{A}=\{X_{1},...,X_{n}\}$ we split the rows of $M$ in an
arbitrary way into sets $A_{i}$ of size $t$. Let $n=st+k,$ where $0\leq k<t$%
. For any permutation $\pi $ and $i\leq s$ 
\begin{equation*}
\mathbf{rank}(X_{\pi (1)},\ldots ,X_{\pi (i)})-i=i(t-1).
\end{equation*}%
For $i\geq s+1$ 
\begin{equation*}
\mathbf{rank}(X_{\pi (1)},\ldots ,X_{\pi (i)})-i=n-i
\end{equation*}%
Therefore, if $k=0$, then the maximum difference is achieved at $i=s=\lceil 
\frac{n}{t}\rceil $. If $k>0$, then the maximum difference is achieved at $%
i=s+1=\lceil \frac{n}{t}\rceil $. So 
\begin{equation*}
\max_{\pi }\min_{i}\left( \mathbf{rank}(X_{\pi (1)},\ldots ,X_{\pi
(i)})-i\right) =n-\lceil \frac{n}{t}\rceil .
\end{equation*}%
The theorem follows.
\end{proof}

\section{Lower and Upper Bounds for Systems over $GF(2)$}

In this section we focus on the binary field, the field most important for
cryptographic applications. We start with an upper bound.

\begin{theorem}
For\/ $n$ sufficiently large, $F_{2}(n,n)\leq \frac{n}{2}-\frac{1}{8}\log
_{2}n.$
\end{theorem}

\begin{proof}
Let $n$ be a fixed natural number. We choose $t$ such that 
\begin{equation*}
4^{t}2t+2\leq n<4^{t+1}2(t+1)+2
\end{equation*}

Then, for $n$ sufficiently large, $t>\frac{1}{4}\log _{2}n$. Let $\mathcal{A}%
=\{X_{1},...,X_{n}\}$ be a collection of sets $X_{i}=\{\mathbf{v}_{i}\mathbf{%
,w}_{i}\}$ comprising two binary vectors of length $n.$

Let $k$ be the largest number such that there exists sets $%
X_{i_{1}},...,X_{i_{k}}$ with $\mathbf{rank}((\dbigcup%
\limits_{j=1}^{k}X_{i_{j}})=2k.$ Clearly, $k\leq \frac{n}{2},$ and 
\begin{equation}
\mathbf{rank}((\dbigcup\limits_{j=1}^{k}X_{i_{j}}\cup X_{s})<2k+2  \label{c}
\end{equation}%
for all $s\in \{1,...,n\}-\{i_{1},...,i_{k}\}.$ For simplicity, assume that $%
\{i_{1},...,i_{k}\}=:\{1,...,k\}.$

We will consider two cases. First, let $k<\frac{n}{2}-\frac{1}{8}\log _{2}n.$
Then, for each $1\leq s\leq k,$ it is 
\begin{equation*}
\max\limits_{s}\mathbf{rank}(\dbigcup\limits_{j=1}^{s}X_{j})-s=2s-s\leq k.
\end{equation*}%
By definition of $k$ we have that for all $s>k,$ there is in $X_{i}$ at
least one vector that is a linear combination of vectors in $%
\dbigcup\limits_{i=1}^{k}X_{i}.$ Thus we get%
\begin{equation*}
\max\limits_{s}\mathbf{rank}(\dbigcup\limits_{j=1}^{s}X_{j})-s\leq
2k+(s-k)-s=k.
\end{equation*}%
Hence in this case $\Delta (\mathcal{X)}=:\min\limits_{\pi }\max\limits_{s}(%
\mathbf{rank}(\dbigcup\limits_{j=1}^{s}X_{j})-s)\leq k<\frac{n}{2}-\frac{1}{8%
}\log _{2}n.$

Now let $k\geq \frac{n}{2}-\frac{1}{8}\log _{2}n.$ As mentioned above, by
definition of $k,$ for each $i>k,$ there is in $X_{i}$ at least one vector
that is a linear combination of vectors in $\dbigcup\limits_{i=1}^{k}X_{i}.$
Let $D$ be a multiset that contains one such vector from each of sets $%
X_{i},i=k+1,...,n$. We note that some vectors might occur in $D$ more than
once. Hence, $\left\vert D\right\vert \geq n-k\geq \frac{n}{2},$ and each
vector in $D$ is a linear combination of vectors in $\dbigcup%
\limits_{i=1}^{k}X_{i}=\dbigcup\limits_{j=1}^{k}\{\mathbf{v}_{i},\mathbf{w}%
_{i}\}.$ To each vector $\mathbf{u}$ in $D$ we assign a $k$-tuple $T_{%
\mathbf{u}}=(x_{1},...,x_{k}),$ where $x_{i}\in \{\alpha ,\beta ,\gamma
,\delta \}$, such that $x_{i}=\alpha $ if neither of $\mathbf{v}_{i},\mathbf{%
w}_{i}$ occurs in the expression of $\mathbf{u}$ as a linear combination of
elements from $\dbigcup\limits_{i=1}^{k}X_{i},$ $x_{i}=\beta $ if $\mathbf{v}%
_{i}$ is there but $\mathbf{w}_{i}$ is not, $x_{i}=\gamma $ if $\mathbf{w}%
_{i}$ is there but $\mathbf{v}_{i}$ is not, and finally, $x_{i}=\delta $ if
both $\mathbf{v}_{i},\mathbf{w}_{i}$ occur there.

As $n\geq 4^{t}2t+2$ there are in $D$ at least $4^{t}t+1$ vectors. Of these
vectors at least $4^{t-1}t+1$ have the same first component, and of the
latter $4^{t-1}t+1$ vectors at least $4^{t-2}t+1$ coincide in the first two
components, etc. Thus, there are at least $t+1$ vectors, say $\mathbf{y}_{0}%
\mathbf{,y}_{1},...,\mathbf{y}_{t},$ that coincide in the first $t$
components. For each $i=1,...,t,$ we get $\mathbf{y}_{i}=\mathbf{y}_{0}%
\mathbf{+x},$ where $\mathbf{x}$ is a linear combination of vectors in $%
X_{t+1},...,X_{k}$. Without loss of generality, assume $\mathbf{y}_{i}\in
X_{k+1+i}$ for $i=0,1,...,t.$ Let $\pi $ be a permutation $%
(t+1,...,k+t+1,1,...,t,k+t+2,...,n).$ We are going to show that, for all $%
1\leq s\leq n,$ 
\begin{equation*}
\mathbf{rank}(\dbigcup\limits_{i=1}^{s}X_{\pi (i)})-s\leq \frac{n}{2}-\frac{t%
}{2}.
\end{equation*}

Consider a family $\mathcal{B}=\{B_{1},...,B_{n}\}$ that satisfies (\ref{c}%
), i.e., for all $s>k$

\begin{equation*}
\mathbf{rank}(\dbigcup\limits_{j=1}^{k}B_{j})=2k,\text{ and }\mathbf{rank}%
(\dbigcup\limits_{j=1}^{k}B_{j}\cup B_{s})<2k+2.
\end{equation*}%
Further, there is $z_{i}\in B_{k+1+i}$ for $i=0,...,t,$ such that, for all $%
i $ $=1,...,n\,\ $it is $z_{i}=z_{0}+x_{i},$ where $x_{i}$ is a linear
combination of vectors $\dbigcup\limits_{j=t+1}^{k}B_{j},$ such that 
\begin{equation}
\mathbf{rank}(\dbigcup\limits_{i=t+1}^{k+t+1}B_{i})=2k-2t+2+t=2k-t+2.
\label{a}
\end{equation}%
In addition, the collection $\mathcal{B}$ satisfies, for all $1\leq s\leq t,$%
\begin{equation}
\mathbf{rank}(\dbigcup\limits_{i=t+1}^{k+t+1}B_{i}\cup
\dbigcup\limits_{i=1}^{s}B_{i})=\min \{n,\mathbf{rank}(\dbigcup%
\limits_{i=t+1}^{k+t+1}B_{i})+2s\}.  \label{b}
\end{equation}%
We note that such a collection $\mathcal{B}$ exists to each collection $%
\mathcal{X}$ and it is not difficult to construct it. For each $1\leq s\leq
n,$ we have

\begin{equation*}
\mathbf{rank}(\dbigcup\limits_{i=1}^{s}X_{\pi (i)})\leq \mathbf{rank}%
(\dbigcup\limits_{i=1}^{s}B_{\pi (i)}),
\end{equation*}%
which in turn implies

\begin{equation*}
\mathbf{rank}(\dbigcup\limits_{i=1}^{s}X_{\pi (i)})-s\leq \mathbf{rank}%
(\dbigcup\limits_{i=1}^{s}B_{\pi (i)})-s.
\end{equation*}

Thus, $\Delta ($\QTR{cal}{X}$\mathcal{)}$ $\leq \Delta \mathcal{(B)}$. To
finish the proof we show that, for all $1\leq s\leq n,$ it is $\mathbf{rank}%
(\dbigcup\limits_{i=1}^{s}B_{\pi (i)})-s\leq \frac{n}{2}-\frac{t}{2}.$ For $%
1\leq s\leq k+t+1,$ by (\ref{a}) and (\ref{b}), 
\begin{equation*}
\mathbf{rank}(\dbigcup\limits_{i=1}^{s+1}B_{\pi (i)})\geq \min \{n,\mathbf{%
rank}(\dbigcup\limits_{i=1}^{s}B_{\pi (i)})+1\}.
\end{equation*}%
In addition, as $2(k+t+1)\geq 2(\frac{n}{2}-\frac{1}{8}\log _{2}n+\frac{1}{4}%
\log _{2}n+1)\geq n,$ it is 
\begin{equation*}
\mathbf{rank}(\dbigcup\limits_{i=1}^{k+t+1}B_{\pi (i)})=n.
\end{equation*}%
Therefore, for each $s\leq n-1,$ we have 
\begin{equation*}
\mathbf{rank}(\dbigcup\limits_{i=1}^{s+1}B_{\pi (i)})\geq \min \{n,\mathbf{%
rank}(\dbigcup\limits_{i=1}^{s}B_{\pi (i)})+1\}.
\end{equation*}%
Therefore, $\mathbf{rank}(\dbigcup\limits_{i=1}^{s}B_{\pi (i)})-s$ is a
non-decreasing function in $s$ and it achieves its maximum at the smallest
value of $s$ where $\mathbf{rank}(\dbigcup\limits_{i=1}^{s}B_{\pi (i)})=n.$
By (\ref{a}) and (\ref{b}) we get that the maximum value is achieved at $%
s=k+1+s_{0}$ where 
\begin{equation*}
s_{0}=\left\lceil \frac{1}{2}(n-2k-2+t)\right\rceil \leq t
\end{equation*}%
as we assume in this case that $k\geq \frac{n}{2}-\frac{1}{8}\log _{2}n.$

This in turn implies 
\begin{equation*}
\Delta \mathcal{(B)=}\mathbf{rank}(\dbigcup\limits_{i=1}^{s_{0}}B_{i}\cup
\dbigcup\limits_{i=t+1}^{k+t+1}B_{i})-(s_{0}+k+1)=n-(s_{0}+k+1)=\frac{n}{2}-%
\frac{t}{2}.
\end{equation*}

Therefore $\Delta ($\QTR{cal}{X}$\mathcal{)}\leq \Delta (\mathcal{B)\leq }%
\frac{n}{2}-\frac{\log _{2}n}{8}$ as $t>\frac{1}{4}\log _{2}n;$ i.e., $%
F_{2}(n,n)\leq \frac{n}{2}-\frac{1}{8}\log _{2}n$.
\end{proof}

\bigskip

Now we state two linear lower bounds on $F_{2}(n,n).$

\begin{theorem}
For all\/ $n$ sufficiently large, $F_{2}(n,n)\geq \frac{n}{9.0886}.$
\end{theorem}

\begin{proof}
First we note that it is easy to check that $c=0.2200557288$ satisfies the
inequality 
\begin{equation}
H_{2}(c/2)<1/2,
\end{equation}%
where $H_{2}$ is the binary entropy function.

To prove the statement we will show that for all $n$ sufficiently large
there exists a matrix $H$ over $GF(2)$ of size $2n\times n$ with the
property that any $d-1$ of its rows are linearly independent, where $%
d=\left\lceil cn\right\rceil $. The matrix $H$ will be constructed from a
parity-check matrix of a suitable linear code.

Let $N=2n$ and let $r$ be the natural number such that 
\begin{equation*}
2^{r-1}\leq \sum_{i=1}^{d-2}\binom{N-1}{i}<2^{r},
\end{equation*}%
It is well known, see e.g. Corollary 9, Chapter 10 in \cite{SM}, that for $%
0<\lambda <\frac{1}{2},$ it is%
\begin{equation}
\dsum\limits_{k=0}^{\lambda n}\binom{n}{k}\leq 2^{nH_{2}(\lambda )}.
\label{4}
\end{equation}%
By definition of $c$ we have $d-2\leq \frac{c}{2}(N-1).$ Applying (\ref{4})
and using $\frac{d-2}{N-1}\leq \frac{c}{2}$ we arrive at 
\begin{equation*}
2^{r-1}\leq \sum_{i=1}^{d-2}\binom{N-1}{i}\leq 2^{(N-1)H_{2}\left( \frac{d-2%
}{N-1}\right) }\leq 2^{(N-1)H_{2}\left( \frac{c}{2}\right) }\leq 2^{n-1}
\end{equation*}%
as $(N-1)H_{2}\left( \frac{c}{2}\right) \leq n-1$ for $n$ sufficiently
large. This in turn implies $r\leq n$ for those $n$. Therefore, by the
Gilbert-Varshamov bound, Theorem 12, Chapter 1 in \cite{SM}, there is a
binary linear code of length $N$ with at most $r$ parity checks, and the
minimum distance at least d. The theorem is proved by an explicit
construction of an $N\times r$ binary matrix $P$ such that no $d-1$ rows are
linearly dependent. To get a desired matrix $H$ we expand $P$ by arbitrary $%
n-r$ columns, and split rows of $H$ into a collection \QTR{cal}{X}$%
=\{X_{1},..,X_{n}\}$ of $n$ pairs of vectors. Any 
\begin{equation*}
\lfloor (d-1)/2\rfloor
\end{equation*}%
of such pairs constitute a set of $2\lfloor (d-1)/2\rfloor $ linearly
independent vectors. For $n$ sufficiently large we get 
\begin{equation*}
\min_{\pi }\max_{k}(\mathbf{rank}\left( \bigcup_{i=1}^{k}X_{\pi (i)}\right)
-k)\geq \lfloor (d-1)/2\rfloor \geq \frac{cn-2}{2}\geq \frac{n}{9.08861}.
\end{equation*}%
The proof is complete.
\end{proof}

\medskip

At the moment we do not have a conjecture about the asymptotic rate of
growth of the function $F_{2}(n,n).$ To indicate the difficulty of the
problem we present a family exhibiting that a linear lower bound on $%
F_{2}(n,n)$ can be obtained even by a very special system.

\begin{theorem}
For sufficiently large\/ $n\,,$ there is a positive constant\/ $c$ and a
family\/ $\mathcal{X=}\{X_{1},...,X_{n}\}$ of pairs of binary vectors, where
for all\/ $i\in \lbrack n],$ $\left\vert X_{i}\right\vert =2,$ and\/ $X_{i}$
contains a unit vector and a vector with exactly two non-zero coordinates,
such that 
\begin{equation*}
\min_{\pi }\max_{1\leq k\leq n}(\mathbf{rank}\!\!\left(
\dbigcup\limits_{i=1}^{k}X_{\pi (i)}\right) -k)\geq cn,
\end{equation*}%
where the minimum runs over all permutations on\/ $[n].$
\end{theorem}

\begin{proof}
To simplify notation, we construct a family of $3n$ sets of vectors, rather
than just $n$ of them. That is, the vector space $V$ has dimension $3n$ and $%
\mathcal{X}=\{X_{1},...,X_{3n}\}$. We partition the set of $3n$ linearly
independent unit vectors into three sets: 
\begin{equation*}
A=\{a_{1},\dots ,a_{n}\},\quad B=\{b_{1},\dots ,b_{n}\},\quad
C=\{c_{1},\dots ,c_{n}\}.
\end{equation*}%
As in the proof of Theorem 14 in \cite{HT}, we select two permutations $%
\sigma $ and $\tau $ on $[n]$ at random, uniformly, and independently; that
is, any permutation on $[n]$ coincides with each of $\sigma $ and $\tau $
with probability 1/n!, and any ordered pair of permutations of $[n]$
coincides with $(\sigma ,\tau )$ with probability $(1/n!)^{2}$. The
following fact was proved in \cite{HT}, with explicit values\footnote{%
One of the goals in \cite{HT} was to find as good estimates on $c^{\prime
\prime }$ as possible, but in the present paper we only aim at proving that
a positive constant exists as lower bound.} of the constants $c^{\prime }$, $%
c^{\prime \prime }$, and $q$:

\begin{itemize}
\item[$(*)$] There exist positive constants $c^{\prime },c^{\prime \prime
},q $ such that the following property holds with probability at least $q$:
in every ordering of the family of 3-element (multi)sets 
\begin{equation*}
\left\{\{i, \sigma(i), \tau(i)\} \mid i = 1,\dots , n\right\}
\end{equation*}
the union of the first $\lfloor c^{\prime }n \rfloor$ sets has cardinality
at least $\lfloor c^{\prime }n \rfloor + c^{\prime \prime }n$.
\end{itemize}

We now consider the collection of sets of vectors $\mathcal{X}%
=\{X_1,\dots,X_{3n}\}$ over $A\cup B\cup C$ where, for $i = 1,\dots , n$,

$X_{3i-2} = \{ a_i , b_i+c_i \} , ~~ X_{3i-1} = \{ a_i , b_i+c_{\sigma(i)}
\} , ~~ X_{3i} = \{ a_i , b_i+c_{\tau(i)} \} . $

\noindent We will prove that this $\mathcal{X}$ satisfies the stated
inequality of the theorem (approximately for $c=\frac12 c^{\prime \prime }$)
with positive probability; this immediately implies that there exists a
suitable choice of $\mathcal{X}$.

First we observe some properties of the auxiliary random bipartite
(multi)graph $H$, constructed from the pair $(\sigma ,\tau )$, which has the
vertex set $V(H)=B\cup C$ and the edge set 
\begin{equation*}
E(H)=E_{0}\cup E_{\sigma }\cup E_{\tau }
\end{equation*}%
where 
\begin{equation*}
E_{0}=\{b_{i}c_{i}\mid i=1,\dots ,n\},~~E_{\sigma }=\{b_{i}c_{\sigma
(i)}\mid i=1,\dots ,n\},~~E_{\tau }=\{b_{i}c_{\tau (i)}\mid i=1,\dots ,n\}.
\end{equation*}%
We call an edge of $H$ a 0-edge, or $\sigma $-edge, or $\tau $-edge, if it
is in $E_{0}$, or $E_{\sigma }$, or $E_{\tau }$, respectively.\bigskip

\noindent \emph{Claim 1.} For every $\ell\in\mathbb{N}$ there exists a
constant $d_\ell$ such that the expected number of cycles of length $2\ell$
in $H$ is at most $d_\ell$.

\medskip

\noindent \emph{Proof.} Consider any cycle, say $C^*$, of length $2\ell$ in $%
H$. Quantitatively it can be associated with a triplet $(\ell_0,\ell_{%
\sigma},\ell_{\tau})$ which represents that $C^*$ has $\ell_0$ 0-edges, $%
\ell_{\sigma}$ $\sigma$-edges, and $\ell_{\tau}$ $\tau$-edges. For every
fixed $\ell$ there exist a bounded number of cycle types with respect to the
positions of the three different kinds of edges --- the number of
possibilities for prescribing the edge types around $C^*$ is clearly smaller
than $3\cdot 2^{2\ell}$ as each of $0,\sigma,\tau$ must correspond to a
matching, but the actual exact number is unimportant for our purpose. There
are ${\binom{n}{\ell}} = O(n^\ell)$ ways to specify the set $V(C^*)\cap B$,
and those $\ell$ vertices can occur in $\ell!/2$ different orders along $C^*$%
. Once an order and a reference vertex of $C^*$ are fixed, they specify the
positions of 0-edges, and so they also determine $\ell_0$ vertices of $C^*$
in the vertex class $C$. Thus, the independent and uniform random selection
of $\sigma$ and $\tau$ implies that there are ${\binom{n-\ell_0}{\ell-\ell_0}%
} = O(n^{\ell-\ell_0})$ ways to choose the other $\ell-\ell_0$ vertices of $%
C^*$ (and $(\ell-\ell_0)!$ ways to permute them). Consequently the number of
choices for $V(C^*)$ is $O(n^{2\ell-\ell_0})$. For any fixed selection of
those $2\ell$ vertices with their fixed order and specified positions of the
edge types around the cycle, the probability that the prescribed edges are
present in $H$ is equal to $\left( \prod_{j=0}^{\ell_{\sigma}-1} (n-j) \cdot
\prod_{k=0}^{\ell_{\tau}-1} (n-k) \right)^{-1} =
O((n^{\ell_{\sigma}+\ell_{\tau}})^{-1}) = O(n^{-(2\ell-\ell_0)})$ as $n$
gets large (where the `hidden coefficient' in `$O$' depends on~$\ell$). The
expected number of cycles of a given length is not larger than the product
of the probability and the number of possible choices, thus it does not
exceed a suitably chosen constant $d_{\ell}$. \quad $\Diamond$

\bigskip

Let us fix now a positive constant $q$ as in $(*)$.

\bigskip

\noindent \emph{Claim 2.} For every $\ell\in\mathbb{N}$ it has probability
at least $q/2$ that $(*)$ holds simultaneously with the property that, for 
\underline{every} $k\le\ell$, the number of cycles of length $2k$ in $H$ is
at most $2\ell d_k/q$.

\medskip

\noindent \emph{Proof.} It is an elementary fact in probability theory that
every positive-valued random variable with expectation $\mathbb{E}(\xi)$
satisfies the inequality $\mathbb{P}(\xi>s\cdot\mathbb{E}(\xi))<1/s$ for
every $s>1$. Applying this for $s=2\ell/q$ and $k=1,\dots,\ell$, the
assertion follows. \quad $\Diamond$

\bigskip

Returning to the proof of the theorem, we combine $(*)$ with Claim 2 and
observe that the following event has probability at least $q/2$:

\begin{itemize}
\item[$(**)$] If $\ell $ is fixed (arbitrarily) and $n$ is sufficiently
large (with respect to $\ell $) then the number of cycles of length at most $%
2\ell $ in $H$ is bounded from above by some constant $f(\ell )$; moreover,
each $Y\subset B$ with $|Y|=\lfloor c^{\prime }n\rfloor $ has at least $%
\lfloor c^{\prime }n\rfloor +c^{\prime \prime }n$ neighbors in $C$.
\end{itemize}

So, we choose $\mathcal{X}$ (determined by a suitable choice of $\sigma$ and 
$\tau$) accordingly, and assume from now on that $H$ satisfies $(**)$.

Let $\pi $ be any permutation of $\{1,\dots ,3n\}$, say $\pi =(i_{1},\dots
,i_{3n})$, and consider the sequence $X_{i_{1}},\dots ,X_{i_{3n}}$ of sets
of vectors. For any $k=1,\dots ,3n$, in the rest of this proof, with a
slight abuse of notation let $b_{\pi (k)}$ mean the vertex $b_{\lceil
i_{k}/3\rceil }$, that corresponds to the vector in the $B$-component of $%
X_{k}$. In order to define $c_{\pi (k)}$ with an analogous meaning, we write 
$k$ in the form $k=3i-3+j$ where $j\in \{1,2,3\}$. Then let $c_{\pi
(k)}=c_{i}$ if $j=1$, $c_{\pi (k)}=c_{\sigma (i)}$ if $j=2$, and $c_{\pi
(k)}=c_{\tau (i)}$ if $j=3$. Now we set 
\begin{equation*}
B^{\leq k}=\{b_{\pi (j)}\mid 1\leq j\leq k\}.
\end{equation*}%
%
%
%
%
That is, $B^{\leq k}$ 
represents those unit vectors in $B$ 
which have a contribution to the first $k$ sets of vectors. 

This also identifies the pairs $b_{\pi(j)}c_{\pi(j)}$ for $j=1,\dots ,k$,
which we may view as the edges of a bipartite (multi)graph $H_k$. Denoting
by $k_1$, $k_2$, and $k_3$ the number of vertices of $H_k$ in $B$ which have
degree 1, 2, and 3, respectively, the equality $k=k_1+2k_2+3k_3$ is valid.
If $k=1$ then $k_2=k_3=0$, and if $k=3n$ then $k_3=n$; moreover, the sum $%
k_2+k_3$ either remains unchanged or increases by exactly 1 when we increase 
$k$ by 1. Thus, we can and will fix a value $k$ for which $k_2+k_3=\lfloor
c^{\prime }n \rfloor$ holds in the graph $H_k$, where the positive constant $%
c^{\prime }$ is the one from $(*)$ and $(**)$.

We consider two modifications of $H_k$. The first one, denoted by $H_k^-$,
is the subgraph of $H_k$ that belongs to the subsequence obtained by
removing those $k_1$ sets $X_{i_j}$ which correspond to the degree-1
vertices of $H_k$. The second one, denoted by $H_k^+$, is the supergraph of $%
H_k^-$ obtained by inserting those $k_2$ edges of $H$ which correspond to
later sets $X_{i_j}$ in which a degree-2 vertex of $H_k$ occurs. Hence both $%
H_k^-$ and $H_k^+$ have $k_2+k_3$ vertices in $B$; moreover, the number of
edges in $H_k^-$ is $2k_2+3k_3$, and the number of edges in $H_k^+$ is $%
3k_2+3k_3$.

Due to the partition $(A,B\cup C)$ concerning $\mathcal{X}$, the rank is 
a sum of two numbers, namely of the rank in $A$ and in $B\cup C$. In $A$ it
is just $k_1+k_2+k_3$, independently of the situation in $B\cup C$.
Moreover, 
linear independence for the vectors $b_{\pi(j)} + c_{\pi(j)}$ inside $B\cup
C $ 
equivalently means that the corresponding subgraph is cycle-free. 
Thus, writing $comp(G) $ to denote the number of connected components in a
graph $G$, we have 
\begin{equation}  \label{rankeqn}
\mathbf{rank}\!\! \left(\dbigcup\limits_{i=1}^{k}X_{\pi (i)}\right) =
(k_1+k_2+k_3) + (|V(H_k)|-comp(H_k)) .
\end{equation}
It will simplify the computation if we modify the formula from $H_k$ to $%
H_k^+$ as 
\begin{eqnarray}  \label{HtoH+}
|V(H_k)|-comp(H_k) & = & k_1 + |V(H_k^-)|-comp(H_k^-)  \notag \\
& \ge & k_1 + |V(H_k^+)| - comp(H_k^+) - k_2
\end{eqnarray}
which is valid because the $k_1$ pendant edges are not involved in any
cycles, and each of the $k_2$ additional edges of $E(H_k^+)\setminus
E(H_k^-) $ may or may not increase the rank by 1 (at most), but they can
never decrease it. The combination of (\ref{rankeqn}) and (\ref{HtoH+})
yields 
\begin{equation}  \label{rank-H+}
\mathbf{rank}\!\! \left(\dbigcup\limits_{i=1}^{k}X_{\pi (i)}\right) \ge
2k_1+k_3 + |V(H_k^+)| - comp(H_k^+) ,
\end{equation}
so the problem is reduced to finding an appropriate lower bound on the
difference $|V(H_k^+)| - comp(H_k^+)$.

We are going to give two kinds of lower bounds, their combination will prove
the theorem. First, recall that there are $k_2+k_3$ vertices in $B^{\le k}$
which have degree at least 2 in $H_k$. All of them have degree 3 in $H_k^+$,
therefore we can apply $(**)$ and obtain 
\begin{equation}  \label{low-cn}
|V(H_k^+)| \ge 2k_2+2k_3+c^{\prime \prime }n .
\end{equation}
For a different lower bound, we partition $H_k^+$ into the following three
parts:

\begin{itemize}
\item $H^1$ -- tree components with at most $2\ell$ vertices for an integer $%
\ell$ to be defined later,

\item $H^2$ -- components containing at least one cycle and having at most $%
2\ell$ vertices,

\item $H^3$ -- components with more than $2\ell$ vertices.
\end{itemize}

For $j=1,2,3$ let us denote $k_{\langle j\rangle} := |B\cap V(H^j)|$. Then,
in particular, we have 
\begin{equation*}
k_{\langle 1\rangle} + k_{\langle 2\rangle} + k_{\langle 3\rangle} = k_2 +
k_3 .
\end{equation*}

If a component $F$ of $H_k^+$ is a tree, then it satisfies the equality $%
|C\cap V(F)| = 2 |B\cap V(F)| + 1$; we shall use this fact for the
components in $H^1$. Moreover, due to the presence of 0-edges $b_ic_i$, the
subgraph $H^2\cup H^3$ has at least as many vertices in $C$ as in $B$.
Consequently, since $k_{\langle 1\rangle} \ge comp(H^1)$, we obtain: 
\begin{eqnarray}  \label{low-comp}
|V(H_k^+)| & \ge & 3k_{\langle 1\rangle} + comp(H^1) + 2(k_{\langle
2\rangle} +k_{\langle 3\rangle} )  \notag \\
& \ge & 2k_2 + 2k_3 + 2 comp(H^1) .
\end{eqnarray}
Taking the arithmetic means on both sides of the equations (\ref{low-cn})
and (\ref{low-comp}) yields: 
\begin{equation}
|V(H_k^+)| \ge 2k_2+2k_3+ {\textstyle\frac12} c^{\prime \prime 1}) .  \notag
\end{equation}
Substituting this formula into (\ref{rank-H+}), moreover recalling that $%
k=k_1+2k_2+3k_3$ and noting that $comp(H_k^+) = comp(H^1) + comp(H^2) +
comp(H^3)$, the following lower bound is derived: 
\begin{equation}  \label{rank-aver}
\mathbf{rank}\!\! \left(\dbigcup\limits_{i=1}^{k}X_{\pi (i)}\right) \ge k +
k_1 + {\textstyle\frac12} c^{\prime \prime 2}) - comp(H^3) .
\end{equation}
The property concerning short cycles, as described in $(**)$, immediately
implies 
\begin{equation}
comp(H^2) \le f(\ell)  \notag
\end{equation}
and, since the components of $H^3$ are large by definition, due to the
choice $k_2+k_3=\lfloor c^{\prime }n \rfloor$ we also have 
\begin{equation}  \label{utso}
comp(H^3) < \frac{c^{\prime }n+n}{2\ell} < \frac{n}{\ell} .  \notag
\end{equation}
Choosing $\ell$ large (e.g., $\ell=\lceil 3/c^{\prime \prime }\rceil$), the
latter two inequalities together with~(\ref{rank-aver}) complete the proof.
\end{proof}

\bigskip

\noindent \textbf{Acknowledgement.} The authors are indebted to Noga Alon
for discussions on expanders and on probabilistic methods, which lead to an
improvement of the lower bound in Theorem \ref{bounds}, and to \O yvind
Ytrehus for a discussion on the Gilbert-Varshamov bound. 

\end{document}